\documentclass[9pt, a4paper, reqno]{amsart}
 
 \usepackage[a4paper,margin=1.20in]{geometry}

\usepackage{graphicx}
 
\usepackage[pdftex, colorlinks=true, urlcolor=blue, linkcolor=blue, citecolor=blue, pdfstartview=FitH]{hyperref}

\usepackage{epsfig,color}
 \usepackage{hyperref}
 \usepackage{float}

\usepackage{mathrsfs}

  \usepackage{setspace}

\newtheorem{thm}{Theorem}

\newtheorem{cor}[thm]{Corollary}

\theoremstyle{definition}

\newtheorem{exmp}{Example}

\title[length-minimizing level curves]{Length-minimizing level curves via calibrations}

\author{Kwok-Kun Kwong}
\address{School of Mathematics and Applied Statistics, University of Wollongong, NSW 2522, Australia}
\email{kwongk@uow.edu.au}

\author{Hojoo Lee}
\address{Department of Mathematics and Institute of Pure and Applied Mathematics, Jeonbuk National University,
Jeonju 54896, Korea}
\email{compactkoala@gmail.com, kiarostami@jbnu.ac.kr}

\begin{document}

 \maketitle

\begin{abstract}
We present an elementary criterion to show the length-minimizing property of geodesics for a large class of conformal metrics. In particular,  
we prove the length-minimizing property of level curves of harmonic functions and the length-minimizing property of a family of the conic sections with the eccentricity $\varepsilon$ in the upper half plane endowed with the conformal metric $ \left( {\varepsilon}^{2} + \frac{1}{\;{y^2} \;} \right) \left(dx^{2} + dy^{2} \right)$. 
\end{abstract}
 


 
 Inspired by two neat proofs due to G. Lawlor, briefly sketched in \cite[p. 247-248]{Lawlor1996} (see also his paper \cite{Lawlor1998}), of the length-minimizing property of the cycloid in the lower half plane endowed with the conformal metric $ \frac{1}{\;  -y \;} \left( dx^2 + dy^2 \right)$, we provide an elementary first order criterion to determine the length-minimizing property of geodesics for a large class of conformal metrics. The key idea is to construct two families of orthogonal level curves by integrating exact differential equations. 
  
 \begin{thm} \label{minimizers} Let $\Omega \subset {\mathbb{R}}^{2}$ denote an open domain. Assume that two ${\mathcal{C}}^{1}$ functions
 $f: \Omega \to \mathbb{R}$ and  $g: \Omega \to \mathbb{R}$ satisfy the following three conditions:
 \begin{equation*}
   \nabla f(x, y) \cdot \nabla g(x, y) =0, \quad
   \Vert \,  \nabla f(x, y) \, \Vert > 0, \quad
    \Vert \,  \nabla g(x, y) \, \Vert > 0. \\
\end{equation*}
The given ${\mathcal{C}}^{1}$ parameterized curve  ${\mathbf{X}}_{*}$  connecting from $( {x}_{1}, {y}_{1}) \in \Omega$ to $( {x}_{2}, {y}_{2})  \in \Omega$ lies in
the level curve  
 \begin{equation*}
 {C}_{*} = \left\{ \;  \left(x, y \right) \in \Omega \subset {\mathbb{R}}^{2}  \; \vert \;   g(x, y)=g( {x}_{1}, {y}_{1})
    =g( {x}_{2}, {y}_{2})  \; \right\}.
\end{equation*}
Then, it becomes the geodesic with respect to the conformal metric ${\Vert \, \nabla f(x, y) \, \Vert}^{2} \left( {dx}^{2} + {dy}^{2} \right)$. Moreover, for any ${\mathcal{C}}^{1}$ parameterized curve $\mathbf{X} \subset \Omega$ connecting from $( {x}_{1}, {y}_{1})$ to $( {x}_{2}, {y}_{2})$, we have the inequality
\begin{equation*}
   \int_{\mathbf{X}} \, \Vert \,  \nabla f(x, y) \, \Vert \; ds  \geq    \int_{{\mathbf{X}}_{*}} \, \Vert \, \nabla f(x, y) \, \Vert \; ds.
\end{equation*}
In other words, the curve ${\mathbf{X}}_{*}$ becomes a weighted length-minimizer with respect to the density $\Vert \,  \nabla f(x, y) \, \Vert$. 
 \end{thm}
 
 \begin{proof}  On the domain $\Omega$, we introduce the unit vector field ${\mathbf{T}}_{*}(x, y)$ and angle ${\Theta}_{*}(x, y)$ by
   \begin{equation} \label{sign}
{\mathbf{T}}_{*}(x, y) =  \pm  \frac{1}{\; \Vert  \;  \nabla f(x,y)   \; \Vert \;} \nabla f(x,y) 
    = \begin{bmatrix}    \;  \cos {\Theta}_{*}(x, y) \; \\   \;  \sin {\Theta}_{*}(x, y) \; 
      \end{bmatrix},
   \end{equation}
   where the sign will be chosen later, and introduce the unit vector field
   \begin{equation*}
  {\mathbf{N}}_{*}(x, y) =  \frac{1}{\; \Vert  \;  \nabla g(x,y)   \; \Vert \;} \nabla g(x,y). 
  \end{equation*}   
 According to the assumption $\nabla f \cdot   \nabla g =0$ on the domain $\Omega$, we find that, on the level curve ${C}_{*}  \subset  \Omega$,   
  \begin{equation*}
   {\mathbf{T}}_{*}(x, y) \cdot {\mathbf{N}}_{*}(x, y)  =0.
  \end{equation*}   
We choose the sign in (\ref{sign}) so that ${\mathbf{T}}_{*}(x, y)$ is the unit tangent vector field on the oriented curve ${\mathbf{X}}_{*} \subset {C}_{*}$.
The vector field ${\mathbf{T}}_{*}(x, y)$ is a natural extension of the unit tangent vector field on the curve ${\mathbf{X}}_{*}$ to the whole domain $\Omega$. Along the oriented competing curve ${\mathbf{X}}(s)=(x(s), y(s))$, where $s$ is the arc-length parameter starting from the initial point $( {x}_{1}, {y}_{1})$, we define the unit tangent vector field ${\mathbf{T}}(x(s), y(s))$ and  the angle ${\Theta}(x(s), y(s))$ with
  \begin{equation*}
    \mathbf{T}(x(s), y(s))       =     \begin{bmatrix} 
       \;  \frac{\;d\;}{ds} x(s) \; \\  \; \frac{\;d\;}{ds} y(s) \; 
      \end{bmatrix} =  \begin{bmatrix}
       \;  \cos {\Theta}(x(s), y(s)) \; \\  \;  \sin {\Theta}(x(s), y(s)) \; 
      \end{bmatrix}.
  \end{equation*}
On the competitor ${\mathbf{X}}(s)=(x(s), y(s))$ lying in the domain $\Omega$, we consider the geometric quantity 
 \begin{equation*}
    {\mathbf{T}}_{*}(x(s), y(s)) \cdot      \mathbf{T}(x(s), y(s)) = \cos \left( \; {\Theta}_{*}(x(s), y(s)) -  {\Theta}(x(s), y(s))   \; \right).
\end{equation*}
The line integral of the quantity $\Vert \, \nabla f \, \Vert \, \cos \left(\; {\Theta}_{*} - {\Theta} \;\right)$ along the competitor ${\mathbf{X}}(s)$ is constant:
  \begin{equation*}
   \int_{\mathbf{X}}  \, \Vert \, \nabla f \, \Vert \, \cos \left( {\Theta}_{*} - {\Theta} \right)   \, ds  
  =    \int_{\mathbf{X}} \,  \nabla f \cdot  {\mathbf{T}} \, ds_{{}_{{\mathbf{X}}}}  
  =  \int_{\mathbf{X}} \, df =  \int_{{\mathbf{X}}_{*}} \, df
  =   \int_{{\mathbf{X}}_{*}} \,  \nabla f \cdot  {\mathbf{T}}_{*} \, ds_{{}_{{\mathbf{X}}_{*}}}  
   =   \int_{{\mathbf{X}}_{*}} \, \Vert \, \nabla f \, \Vert \, ds.
\end{equation*}   
It follows from this and the estimation $1 \geq \cos \left( {\Theta}_{*} - {\Theta} \right)$ that 
\begin{equation*}
   \int_{\mathbf{X}} \, \Vert \,  \nabla f(x, y) \, \Vert \; ds  \geq    \int_{\mathbf{X}}  \, \Vert \, \nabla f \, \Vert \, \cos \left( {\Theta}_{*} - {\Theta} \right)   \, ds   =    \int_{{\mathbf{X}}_{*}} \, \Vert \, \nabla f(x, y) \, \Vert \; ds.
\end{equation*}
  \end{proof}
  
  The key idea in the above proof of Theorem \ref{minimizers} is to employ the method of calibrations illustrated in a beautiful paper \cite{HarveyLawson1982} by R. Harvey and H. B. Lawson. For more recent resources, we refer the interested readers to D. Joyce'a book \cite{Joyce2007} and J. Lotay's article \cite{Lotay2020}. Neat expositions on the calibrations, manifolds with density, isoperimetric problems with density are given by F. Morgan \cite{Morgan1988, Morgan2005, Morgan2017}.

    \begin{exmp}[Astroid length density $\sqrt{\, {x}^{\frac{2}{\,3\,}} + {y}^{\frac{2}{\,3\,}}  \,}\,$]  \label{astroid} 
 We consider the one parameter family of (a part of) the astroids in the quadrant $\Omega = \left( 0, \infty \right) \times \left( 0, \infty \right)$:
   \begin{equation*}
 {C}_{\lambda>0} =  \left\{ \;  \left(x, y \right) \in \Omega \; \vert \;  \lambda =  g(x, y) = {x}^{\frac{2}{\,3\,}} + {y}^{\frac{2}{\,3\,}} \; \right\},
\end{equation*}
To construct a calibration function $f(x,y)$ in Theorem \ref{minimizers}, we need to solve the orthogonality condition
       \begin{equation*}
  0 = \nabla g(x,y)   \cdot  \nabla f(x,y)  = \frac{2}{\, 3 \,} \begin{bmatrix}
        \; {x}^{-\frac{1}{\,3\,}} \; \\    \;   {y}^{-\frac{1}{\,3\,}}   \;
      \end{bmatrix} \cdot  \begin{bmatrix}
        \;  f_{x} \; \\    \;   f_{y} \;    
      \end{bmatrix} = \frac{2}{\, 3 \,} \left( \, {x}^{-\frac{1}{\,3\,}} f_{x} + {y}^{-\frac{1}{\,3\,}} f_{y} \, \right).
        \end{equation*}          
Integrating the induced exact differential equation $df = f_{x} dx + f_{y} dy =  {x}^{\frac{1}{\,3\,}} dx - {y}^{\frac{1}{\,3\,}}  dy$       
gives the function $f(x, y) = C \left( \, {x}^{\frac{4}{\,3\,}} - {y}^{\frac{4}{\,3\,}} \, \right)$ for a constant $C \in \mathbb{R}$.   
We take the function $f : \Omega \to \mathbb{R}$ defined by 
          \begin{equation*}
f(x, y)=  \frac{3}{\, 4 \,}  \left(  {x}^{\frac{4}{\,3\,}} - {y}^{\frac{4}{\,3\,}} \right)
        \end{equation*}    
and compute the length density 
$\Vert \, \nabla f(x,y) \, \Vert = \sqrt{\, {x}^{\frac{2}{\,3\,}} + {y}^{\frac{2}{\,3\,}}  \,} =  \sqrt{\,    g(x, y)  \,}$.
We conclude that the astroid ${x}^{\frac{2}{\,3\,}} + {y}^{\frac{2}{\,3\,}}=\lambda$ becomes an weighted length-minimizer   
  with respect to the length density $\sqrt{\, {x}^{\frac{2}{\,3\,}} + {y}^{\frac{2}{\,3\,}}  \,}$.      
  \end{exmp}
  
Example \ref{astroid} can be generalized in various ways. We consider geodesics in the quadrant endowed with the metric 
 $\left(  {x}^{2p} + {y}^{2q} \right) \left(dx^2 + dy^2 \right)$. For each $\lambda \in \mathbb{R}$, we introduce the function ${\Psi}_{\lambda} : \left(0, \infty \right) \to \mathbb{R}$ by  \begin{equation*}
      {\Psi}_{\lambda} \left( t \right) = \begin{cases}
         \frac{1}{\,1-\lambda\,} {t}^{1 - \lambda}, \quad \lambda \neq 1, \\
       \;   \ln t, \qquad \quad \;  \lambda=1.
      \end{cases} 
 \end{equation*}
Given a pair $\left(p, q\right)$ of real constants, take $g(x, y) = {\Psi}_{p}\left(x\right)+{\Psi}_{q}\left(y\right)$ and $f(x, y)= {\Psi}_{-p}\left(x\right)-{\Psi}_{-q}\left(y\right)$. The level curves of the function $g$ are length-minimizers with respect to the density $\Vert \, \nabla f(x,y) \, \Vert = \sqrt{\, {x}^{2p} + 
{y}^{2q}  \,}$.

  \begin{exmp}[Brachistochrone length density $\frac{1}{\; \sqrt{ \, -y\, } \;}$] \label{brachistochrone density} 
  For various solutions and generalizations to Johann Bernoulli's time-minimizing curve problem, we refer to, for instance, \cite{Benson1969, HawsKiser1995, Kuczmarski2015, Lawlor1996}. We present details how to construct the calibration function $f(x,y)$ in the lower half plane with respect to the density $\frac{1}{\; \mathbf{v}\left(y\right) \;}=\frac{1}{\; \sqrt{ \, -y\, } \;}$. Our construction here will be generalized in Corollary \ref{symmetry}. 
  Integrating the first variation formula for the weighted length functional with respect to $\frac{1}{\;  \mathbf{v}\left(y\right) \;} ds$ gives 
 \begin{equation*}
     \frac{1}{\;  \mathbf{v}\left(y\right) \;}  \frac{dx}{\, ds \,}  = \text{constant}.
 \end{equation*}
Indeed, letting $\mathcal{L} \left(y, x, \dot{x} \right) = \frac{1}{\;  \mathbf{v}\left(y\right) \;} \sqrt{\, 1 + {\dot{x}}^{2} \,}$ with $\dot{x}=\frac{dx}{\,dy\,}$, the Euler-Lagrange equation for the weighted length functional with respect to the weighted length element 
       \begin{equation*}
       \frac{1}{\;  \mathbf{v}\left(y\right) \;} ds=\frac{1}{\;  \mathbf{v}\left(y\right) \;} \sqrt{\, {dx}^{2} + {dy}^{2}  \,} = \mathcal{L} \left(y, x, \dot{x} \right) dy
        \end{equation*}
        reads
 \begin{equation*}
  0 =  \frac{\partial \mathcal{L}}{\, \partial x\,} - \frac{d}{\, dy\,} \, \left( \,  \frac{\partial \mathcal{L}}{\, \partial \dot{x}\,}  \, \right) = 
  - \frac{d}{\, dy\,} \, \left( \, \frac{1}{\;  \mathbf{v}\left(y\right) \;}  \frac{\dot{x}}{\,\sqrt{\, 1 + {\dot{x}}^{2} \,}\,}  \, \right),
 \end{equation*}
which guarantees that, for some constant $c \in \mathbb{R}$, 
 \begin{equation*}
 c = \frac{1}{\;  \mathbf{v}\left(y\right) \;}  \frac{\dot{x}}{\,\sqrt{\, 1 + {\dot{x}}^{2} \,}\,} 
 =   \frac{1}{\;  \mathbf{v}\left(y\right) \;}  \frac{dx}{\, ds \,}, \quad 
 \text{or equivalently}, \quad
 0 = \pm dx + \frac{c \mathbf{v}\left(y\right)}{\, \sqrt{\, 1 - c^2 { \mathbf{v}\left(y\right) }^2{}  \,} \,} dy.
 \end{equation*}
  Integrating the first integral of the geodesic equation indicates why the level set formulation of geodesics is natural. 
 In the case when $c=0$, this reduces to $0=dx$, which gives the vertical ray $x=\text{constant}$ as a geodesic. 
From now on, we consider the case $c^2 >0$ and take the plus sign in the above equation. Recall the factor $ \mathbf{v}\left(y\right)= \sqrt{ \, -y\, }$.
Letting $\rho=\frac{1}{\,2 c^2 \,}>0$, it becomes the exact differential equation 
 \begin{equation*}
  dg =  g_{x} dx + g_{y}  dy = dx + \sqrt{\, \frac{ - c^{2} y }{ 1 + c^2 y } \,} dy = dx + \sqrt{\, \frac{ - 2 \rho y }{ 1 + 2 \rho y } \,} dy,
 \end{equation*}
 which gives the level set formulation of the geodesic 
 \begin{equation*}
C = \left\{ \;  \left(x, y \right) \in {\mathbb{R}}^2 \; \vert \;  0 = g(x, y)=  x - \rho \arccos\left(  \; 1 + \frac{y}{\; \rho\;} \; \right) + \rho \sqrt{ 1 - {\left( \; 1 + \frac{y}{\; \rho\;}  \; \right)}^{2} }  \; \right\}.
 \end{equation*}
(The observation that it admits the cycloid path $\left(x, y\right)=\left( \rho \left( t - \sin t \right),  - \rho \left( 1 - \cos t \right) \right)$ will not be used in the minimization part.) We consider the strip $\Omega = \mathbb{R} \times \left(-2\rho, 0\right)$. To obtain the weighted length-minimizing property of the half cycloid $C \cap \Omega$, we find the function $f(x,y)$ satisfying the orthogonality condition
       \begin{equation*}
  0 = \nabla g(x,y)   \cdot  \nabla f(x,y)  = \begin{bmatrix}
        \; 1 \; \\         \;  \sqrt{ \frac{-y}{\; 2 \rho + y \;} } \;
      \end{bmatrix}  \cdot  \begin{bmatrix}
        \;  f_{x} \; \\    \;   f_{y} \;    
      \end{bmatrix} = f_{x} + \sqrt{ \frac{-y}{\; 2 \rho + y \;} }  f_{y}.
        \end{equation*}          
It induces the exact differential equation
          \begin{equation*}
   df = f_{x} dx + f_{y} dy =  - dx +  \sqrt{ \frac{\; 2 \rho + y \;}{-y} }  dy.
        \end{equation*}          
 Taking the calibration function $f: \Omega \to \mathbb{R}$ defined by
  \begin{equation*}
      f(x, y) = \frac{1}{\; \sqrt{\; 2\rho} \;} \left( -x +  \rho \arcsin \left(  \; 1 + \frac{\;  y\;}{\rho} \; \right) -  \rho \sqrt{ 1 - {\left( \; 1 + \frac{y}{\; \rho\;}  \; \right)}^{2} } \right)
   \end{equation*}
 in Theorem \ref{minimizers} gives the weighted length-minimizing property of the half cycloid $C \cap \Omega$ with respect to the length density $\Vert \, \nabla f(x,y) \, \Vert =  \frac{1}{\;\sqrt{\; -y\;} \;}$. One may ask the reason why the calibration function obtained by integrating the orthogonality condition $f_{x}g_{x} + f_{y}g_{y}=0$ gives the desired density $\sqrt{\, {f_{x}}^{2} +  {f_{y}}^{2} \,}=  \frac{1}{\;\sqrt{\; -y\;} \;}$. It is not a coincidence. It is due to the symmetry of the density, as indicated in the proof of Corollary \ref{symmetry}. 
   \end{exmp}
   
  Example \ref{brachistochrone density} can be extended to the case when the density function has a symmetry.
  
 \begin{cor}[Generalized brachistochrone length density $ \frac{1}{\; \mathbf{v}(y) \;}>0$] \label{symmetry} 
 Let $\mathbf{v}(y)>0$ be a ${\mathcal{C}}^{1}$ function. Given a constant $c \in \mathbb{R}$, 
we consider a ${\mathcal{C}}^{1}$ solution $g: \Omega \to \mathbb{R}$ of the exact differential equation
 \begin{equation*}
  dg = g_{x} dx + g_{y} dy = dx + \frac{c \mathbf{v}\left(y\right)}{\, \sqrt{\, 1 - c^2 { \mathbf{v}\left(y\right) }^{2}  \,} \,} \, dy.
 \end{equation*}
Here, we choose a simply connected domain $\Omega \subset {\mathbb{R}}^{2}$ such that $1 - c^2 { \mathbf{v}\left(y\right) }^{2}>0$ on $\Omega$. Then, the above differential equation is well-defined and solvable on $\Omega$. The level curve (assuming that it is non-empty)
 \begin{equation*}
 {C}_{*} = \left\{ \;  \left(x, y \right) \in \Omega  \; \vert \;   g(x, y)=0 \; \right\}
\end{equation*}
becomes a weighted length-minimizing geodesic in $\Omega$ with respect to the density function $\frac{1}{\; \mathbf{v}(y) \;}$. 
 \end{cor}
 
\begin{proof} In Example \ref{brachistochrone density}, we noted that the differential equation 
 \begin{equation*}
dg = dx + \frac{c \mathbf{v}\left(y\right)}{\, \sqrt{\, 1 - c^2 { \mathbf{v}\left(y\right) }^{2}  \,} \,} \, dy
\end{equation*}
 is the first integral of the weighted geodesic equation for the length density $ \frac{1}{\; \mathbf{v}(y) \;}>0$. We observe that the curve ${C}_{*}$ is regular due to the inequality
\begin{equation*}
  {\Vert \, \nabla g(x, y) \, \Vert }^{2} = \frac{1}{\, 1 - c^2 { \mathbf{v}(y) }^{2} \,}>0.
\end{equation*}
We use Poincar\'{e} Lemma to find a ${\mathcal{C}}^{1}$ solution $f: \Omega \to \mathbb{R}$ of the exact differential equation
 \begin{equation*}
  df = f_{x} dx + f_{y} dy =  -c  \,   dx + \frac{\, \sqrt{\, 1 - c^2 { \mathbf{v}\left(y\right) }^{2}  \,} \,}{ \mathbf{v}\left(y\right)}  \, dy.
\end{equation*}
We check the orthogonality condition
         \begin{equation*}
   \nabla g(x,y)   \cdot  \nabla f(x,y)  =  \begin{bmatrix}
      \; 1 \; \\ \;  \frac{c \mathbf{v}\left(y\right)}{\, \sqrt{\, 1 - c^2 { \mathbf{v}\left(y\right) }^{2}  \,} \,}  \;
      \end{bmatrix} \cdot   \begin{bmatrix}
              \; -c \; \\ \; \frac{\, \sqrt{\, 1 - c^2 { \mathbf{v}\left(y\right) }^2{}  \,} \,}{ \mathbf{v}\left(y\right)}  \;
      \end{bmatrix} = -c + c = 0
        \end{equation*}  
and compute the length density
    \begin{equation*}
         \Vert \,  \nabla f(x, y) \, \Vert= \sqrt{ \, {f_{x}}^{2} + {f_{y}}^{2} \, }
      =      \sqrt{ \, {  c   }^{2} +     \frac{\,  \, 1 - c^2 { \mathbf{v}\left(y\right)   \,} \,}{  {\mathbf{v}\left(y\right)}^{2}  }     \, }
         = \frac{1}{\; \mathbf{v}(y) \;}.
   \end{equation*}     
Theorem \ref{minimizers} shows the length-minimizing property of the level curve ${C}_{*}$ for the density $\frac{1}{\; \mathbf{v}(y) \;}$. 
 \end{proof}

We deform the hyperbolic metric in the Poincar\'{e} half-plane to construct a one parameter family of conformal metrics in the upper half plane so that a family of the conic sections are geodesics. In Examples \ref{eccentricity c}, \ref{eccentricity e}, \ref{eccentricity p}, \ref{eccentricity h}, we use Theorem \ref{minimizers} 
or Corollary \ref{symmetry} to show the length-minimizing property of such curves.

  \begin{exmp}[Conic sections with the eccentricity $\varepsilon \geq 0$]  \label{eccentricity} We preview curves in Examples \ref{eccentricity e}, \ref{eccentricity p}, \ref{eccentricity h}. Given a constant $\varepsilon \in \mathbb{R}$, we define the curve 
 \begin{equation*}
   {D}_{\varepsilon} = \left\{ \;  \left(x, y \right) \in {\mathbb{R}}^2   \; \vert \;    1 - \varepsilon x = \sqrt{\, x^2 +y^2 \,}  \; \right\}.
\end{equation*}
The curve ${D}_{\varepsilon}$ becomes a conic section with the eccentricity $\varepsilon$ and the focus $(0, 0)$. The curve ${D}_{1}$ is a parabola $2x+y^2 =1$. For $\varepsilon \neq 1$, translating the curve ${D}_{\varepsilon}$ horizontally gives the curve    
 \begin{eqnarray*}
   {C}_{\varepsilon} &=& \left\{ \;  \left(x, y \right) \in {\mathbb{R}}^2   \; \bigg\vert \;    1 - \varepsilon   \left( x - \frac{\varepsilon}{ \, 1 - {\varepsilon}^{2}  \, } \right) = \sqrt{
   \, {\left( x - \frac{\varepsilon}{ \, 1 - {\varepsilon}^{2} \, } \right)}^{2} +y^2 \, }  \; \right\} \\
    &=& \left\{ \;  \left(x, y \right) \in {\mathbb{R}}^2   \; \bigg\vert \;  {\left(  1 - {\varepsilon}^{2}   \right)}^{2} x^2 + \left(  1 - {\varepsilon}^{2}  \right) y^2  =1  \; \right\}. 
\end{eqnarray*}
 \end{exmp}

   \begin{exmp}[Conic section length density $\sqrt{\, {\varepsilon}^{2} + \frac{1}{\;{y^2} \;} \,}\,$ with $\varepsilon=0$]  \label{eccentricity c}
 We consider the hyperbolic length density $ \frac{1}{\;{y} \;}$. Given two constants $x_{0} \in \mathbb{R}$
   and $\rho>0$, we take the quarter circle ${C}_{*}$ in the strip $\Omega = \mathbb{R} \times \left(0, \rho \right)$:
 \begin{equation*}
 {C}_{*} = \left\{ \;  \left(x, y \right) \in \Omega  \; \vert \;  0 = g(x, y) =  x - x_{0} - \sqrt{\; {\rho}^{2} -  {y}^{2} \;}  \; \right\}.
\end{equation*}
 To obtain the weighted length-minimizing property of the curve ${C}_{*}$ in the domain $\Omega$ 
  with respect to the length density $\Vert \, \nabla f(x,y) \, \Vert =  \frac{1}{\;{y} \;}$, we introduce the function $f: \Omega \to \mathbb{R}$ defined by
  \begin{equation*}
      f(x, y) = - \frac{x}{\;\rho\;} + \sqrt{1 - {\left( \frac{y}{\;\rho \;} \right)}^{2} } - {\tanh}^{-1} \left( \; \sqrt{\, 1 - {\left( \frac{y}{\; \rho \;} \right)}^{2} \,}  \; \right).
   \end{equation*}
The level curves of two functions of $f(x,y)$ and $g(x,y)$ are orthogonal:
         \begin{equation*}
   \nabla g(x,y)   \cdot  \nabla f(x,y)  =  \begin{bmatrix}
        \; 1 \; \\       \;   \frac{y}{\; \sqrt{\;   { \rho}^2 - y^2      \;} \;} \;
      \end{bmatrix} \cdot   \begin{bmatrix}
        \; - \frac{1}{\; \rho \;}   \; \\  \;   \frac{\; \sqrt{\;   { \rho}^2 - y^2      \;} \;}{\rho y}    \;    
      \end{bmatrix} = 0.
        \end{equation*}      
 \end{exmp}

      \begin{exmp}[Conic section length density $\sqrt{\, {\varepsilon}^{2} + \frac{1}{\;{y^2} \;} \,}$ with $\varepsilon \in \left(0, 1\right)$] \label{eccentricity e}   
            Let $\varepsilon \in \left(0, 1\right)$ and $x_{0} \in \mathbb{R}$ be constants.    
        The level curve ${C}_{\varepsilon, x_{0}}$ in the strip $\Omega = \mathbb{R} \times \left(0, \frac{1}{\, 1 - {\varepsilon}^{2} \,} \right)$  defined by 
      \begin{equation*}
        {C}_{\varepsilon, x_{0}} = \left\{ \;  \left(x, y \right) \in \Omega \; \Bigg\vert \;  
       0= g(x,y) = \left( 1 - {\varepsilon}^{2} \right) \left(x - x_{0} \right) - \sqrt{\,1 - \left( 1 - {\varepsilon}^{2} \right) y^2 \,}
        \; \right\}
\end{equation*}
 is a part of the ellipse with the eccentricity $\varepsilon \in \left(0, 1\right)$.
  To obtain the weighted length-minimizing property of the curve $  {C}_{\varepsilon, x_{0}}$ in the domain $\Omega$ 
  with respect to the length density $\Vert \, \nabla f(x,y) \, \Vert = \sqrt{\, {\varepsilon}^{2} + \frac{1}{\;{y^2} \;} \,}$, we introduce the function $f: \Omega \to \mathbb{R}$ defined by
  \begin{equation*}
      f(x, y) = - x +\sqrt{\,1 - \left( 1 - {\varepsilon}^{2} \right) y^2 \,} - {\tanh}^{-1} \left( \; \sqrt{\,1 - \left( 1 - {\varepsilon}^{2} \right) y^2 \,} \; \right).
   \end{equation*}
The level curves of two functions of $f(x,y)$ and $g(x,y)$ are orthogonal:
         \begin{equation*}
   \nabla g(x,y)   \cdot  \nabla f(x,y)  =  \left( 1 - {\varepsilon}^{2} \right) \begin{bmatrix}
        \; 1 \; \\       \;   \frac{  y}{\; \sqrt{\,1 - \left( 1 - {\varepsilon}^{2} \right) y^2 \,}  \;} \;
      \end{bmatrix} \cdot   \begin{bmatrix}
        \; - 1   \; \\  \;  \frac{\; \sqrt{\,1 - \left( 1 - {\varepsilon}^{2} \right) y^2 \,}  \;}  { y}   \;    
      \end{bmatrix} = 0.
        \end{equation*}    
  \end{exmp}

          \begin{exmp}[Conic section length density $\sqrt{\, {\varepsilon}^{2} + \frac{1}{\;{y^2} \;} \,}\,$ with $\varepsilon=1$]       \label{eccentricity p}
         Let $x_{0} \in \mathbb{R}$ be a constant.     
        The level curve ${C}_{x_{0}}$ in the upper-half plane $\Omega =\mathbb{R} \times \left(0, \infty \right)$ defined by 
      \begin{equation*}
        {C}_{x_{0}} = \left\{ \;  \left(x, y \right) \in \Omega \; \Bigg\vert \;  
       0 = g(x,y) =  x - x_{0} + \frac{1}{\,2\,} {y}^{2} 
        \; \right\}
\end{equation*}
 is a part of the parabola with the eccentricity $\varepsilon=1$.
  To obtain the weighted length-minimizing property of the curve $  {C}_{\varepsilon, x_{0}}$ in the upper-half plane $\Omega$ 
  with respect to the length density $\Vert \, \nabla f(x,y) \, \Vert = \sqrt{\,1 + \frac{1}{\;{y^2} \;} \,}$, we introduce the function $f: \Omega \to \mathbb{R}$ defined by
  \begin{equation*}
      f(x, y) = - x + \ln y.
   \end{equation*}
The level curves of two functions of $f(x,y)$ and $g(x,y)$ are orthogonal:
         \begin{equation*}
   \nabla g(x,y)   \cdot  \nabla f(x,y)  =    \begin{bmatrix}
        \; 1 \; \\       \;  y \;
      \end{bmatrix} \cdot   \begin{bmatrix}
        \; - 1   \; \\  \;  \frac{\;1  \;}  { y}   \;    
      \end{bmatrix} = 0.
        \end{equation*}   
  \end{exmp}

        \begin{exmp}[Conic section length density $\sqrt{\, {\varepsilon}^{2} + \frac{1}{\;{y^2} \;} \,}\,$ with $\varepsilon \in \left(1, \infty\right)$]   \label{eccentricity h}    
         Let $\varepsilon \in \left(1, \infty\right)$ and $x_{0} \in \mathbb{R}$ be constants.     
        The level curve ${C}_{\varepsilon, x_{0}}$ in the upper-half plane $\Omega =\mathbb{R} \times \left(0, \infty \right)$ defined by 
      \begin{equation*}
        {C}_{\varepsilon, x_{0}} = \left\{ \;  \left(x, y \right) \in \Omega \; \Bigg\vert \;  
       0 = g(x,y) = \left( 1 - {\varepsilon}^{2} \right) \left(x - x_{0} \right) - \sqrt{\,1 - \left( 1 - {\varepsilon}^{2} \right) y^2 \,}
        \; \right\}
\end{equation*}
 is a part of one branch of the hyperbola with the eccentricity $\varepsilon \in \left(1, \infty\right)$.
  To obtain the weighted length-minimizing property of the curve $  {C}_{\varepsilon, x_{0}}$ in the domain $\Omega$ 
  with respect to the length density $\Vert \, \nabla f(x,y) \, \Vert = \sqrt{\, {\varepsilon}^{2} + \frac{1}{\;{y^2} \;} \,}$, we introduce the function $f: \Omega \to \mathbb{R}$ defined by
  \begin{equation*}
      f(x, y) = - x +\sqrt{\,1 - \left( 1 - {\varepsilon}^{2} \right) y^2 \,} - {\tanh}^{-1} \left( \; \sqrt{\,1 - \left( 1 - {\varepsilon}^{2} \right) y^2 \,} \; \right).
   \end{equation*}
  \end{exmp}

    \begin{exmp}[Grim reaper length density $e^{y}$] 
It is well-known that the grim reaper $ y = - \ln  \left( \cos x \right)$ in $\left(- \frac{\;\pi\;}{2}, \frac{\;\pi\;}{2} \right) \times \left[0, \infty \right)$ is a translating soliton for the curve shortening flow in the Euclidean plane ${\mathbb{R}}^{2}$. We consider the right half of the grim reaper lying in the domain $\Omega = \left(0, \frac{\;\pi\;}{2} \right) \times \left( 0, \rho \right)$: 
   \begin{equation*}
 {C}_{*} =  \left\{ \;  \left(x, y \right) \in {\mathbb{R}}^{2}  \; \vert \;   y = - \ln  \left( \cos x \right), \, x \in \left(0, \frac{\;\pi\;}{2} \right)  \; \right\} =  \left\{ \;  \left(x, y \right) \in  \Omega \; \vert \;  g(x, y)=0  \; \right\},
\end{equation*}
where the function $g : \Omega \to \mathbb{R}$ is defined by 
    \begin{equation*}
g(x, y) = x - \arccos \left(  e^{-y} \right).
\end{equation*}
 To obtain the weighted length-minimizing property of the half grim reaper ${C}_{*}$  in the domain $\Omega$ 
  with respect to the length density $\Vert \, \nabla f(x,y) \, \Vert  = e^{y}$, we prepare the function $f : \Omega \to \mathbb{R}$  defined by 
    \begin{equation*}
  f(x, y) =  x + \sqrt{ \, e^{2y} - 1 \, } - \arccos \left( e^{-y} \right).
 \end{equation*} 
The level curves of two functions of $f(x,y)$ and $g(x,y)$ are orthogonal:
         \begin{equation*}
   \nabla g(x,y)   \cdot  \nabla f(x,y)  =  \begin{bmatrix}
        \; 1 \; \\    \;   \frac{-1}{\; \sqrt{\;  e^{2y} -1   \;} \;}  \;
      \end{bmatrix} \cdot   \begin{bmatrix}
        \;  1 \; \\    \; \sqrt{\;  e^{2y} -1 \;}   \;    
      \end{bmatrix} = 0.
        \end{equation*}  
 \end{exmp}

 \begin{cor}[Level curves of harmonic functions] \label{harmonic level set} Let $p(x,y) + i q(x,y)$ be a holomorphic function such that $ \Vert \,  \nabla p(x,y) \, \Vert >0$ on an open domain $\Omega \subset {\mathbb{R}}^{2}$.  
 We consider the one parameter family $ {\left\{  {C}_{\alpha}  \right\}}_{\alpha \in \mathbb{R}}$ of the level curves of the harmonic functions: 
  \begin{equation*}
 {C}_{\alpha} = \left\{ \;  \left(x, y \right) \in \Omega \subset {\mathbb{R}}^{2}  \; \vert \;  \left( \sin \alpha \right) \, p(x, y) + \left( \cos \alpha \right) \,  q (x, y) = 0 \; \right\}.
\end{equation*}
Assume that the ${\mathcal{C}}^{1}$ parameterized curve  ${\mathbf{X}}_{*}$ connecting from $( {x}_{1}, {y}_{1}) \in  {C}_{\alpha}$ to $( {x}_{2}, {y}_{2})  \in  {C}_{\alpha}$ lies in the level curve ${C}_{\alpha}$. For any ${\mathcal{C}}^{1}$ parameterized curve $\mathbf{X} \subset \Omega$ connecting from $( {x}_{1}, {y}_{1})$ to $( {x}_{2}, {y}_{2})$, we have 
\begin{equation*}
   \int_{\mathbf{X}} \, \Vert \,  \nabla p(x, y) \, \Vert \; ds  \geq    \int_{{\mathbf{X}}_{*}} \, \Vert \, \nabla p(x, y) \, \Vert \; ds.
\end{equation*}
In other words, the curve ${\mathbf{X}}_{*}$ becomes a weighted length-minimizer with respect to the density $\Vert \,  \nabla p(x, y) \, \Vert$. 
 \end{cor}

\begin{proof} We take the holomorphic function $f(x, y)+ig(x,y) = e^{i \alpha} \left( \, p(x,y) + i q(x,y) \, \right)$ in Theorem \ref{minimizers} to deduce the inequality. It follows from the Cauchy-Riemann equations $p_{x}=q_{y}$ and $p_{y}+q_{x}=0$ that
  \begin{equation*}
    \Vert \,  \nabla f(x, y) \, \Vert = \Vert \,  \nabla g(x, y) \, \Vert = \Vert \,  \nabla p(x, y) \, \Vert = \Vert \,  \nabla q(x, y) \, \Vert >0,
   \end{equation*}
  which implies that the level curve ${C}_{\alpha} = \left\{ \;  \left(x, y \right) \in \Omega \subset {\mathbb{R}}^{2}  \; \vert \;  g (x, y) = 0 \; \right\}$ is regular. 
\end{proof}

 \begin{exmp}[Logarithmic spiral length density $\frac{1}{\,\sqrt{\, x^2 +y^2 \,}\,}$]  We work with the polar coordinates $\left(r, \theta \right)$. Let $r_{0}>0$, $\theta_{0} \in \mathbb{R}$, $\alpha \in \left(0, \pi \right)$, $\lambda = - \cot \alpha$ be constants. Taking the holomorphic function $p(x, y)+i q(x, y) = \ln \left( \frac{r}{\, r_{0} \,} \right) + i \left( \theta - \theta_{0} \right)$ in Corollary \ref{harmonic level set} gives the weighted length-minimizing property of the logarithmic spiral $r = r_{0} \, e^{\lambda \left( \theta - {\theta}_{0} \right) }$ with respect to the length density $\Vert \, \nabla p(x,y) \, \Vert  = \frac{1}{\,\sqrt{\, x^2 +y^2 \,}\,}$.  See also Shelupsky's argument \cite[p. 785]{Shelupsky1961} based on a variation of Snell's Law for the radial density.
 \end{exmp}
 

\end{document}